\documentclass[11pt,oneside,reqno]{amsart}
\usepackage[english]{babel}
\usepackage{amsmath,amsfonts,amsthm,amssymb,amsbsy,upref,color,graphicx,amscd,hyperref,enumerate,bbm,comment, algorithm, algorithmic,dsfont}
\hypersetup{colorlinks,breaklinks,
            linkcolor=[rgb]{0,0,0},
            citecolor=[rgb]{0,0,0},
            urlcolor=[rgb]{0,0,0}}

\usepackage{mathtools,url}

\usepackage[active]{srcltx}
\usepackage{cases} 
\usepackage[shortlabels]{enumitem}
\usepackage{cancel}
\usepackage[capitalise, nameinlink, noabbrev]{cleveref} 

\newcommand{\norm}[2]{\left\Vert {#1} \right\Vert_{#2}}

\DeclareMathOperator{\dive}{div}

\def\eps{{\varepsilon}}

\def\O{\Omega}
\def\R{\mathbb{R}}

\def\F{\mathcal{F}}
\def\HH{\mathcal{H}}

\newcommand{\be}{\begin{equation}}
\newcommand{\ee}{\end{equation}}

\newcommand{\loc}{\mathrm{loc}}

\numberwithin{equation}{section}
\theoremstyle{plain}
\newtheorem{theo}{Theorem}[section]

\newtheorem{lemma}[theo]{Lemma}

\newtheorem{prop}[theo]{Proposition}

\newtheorem{definition}[theo]{Definition}
\theoremstyle{remark}
\newtheorem{remark}[theo]{Remark}

\def\XXint#1#2#3{{\setbox0=\hbox{$#1{#2#3}{\int}$ }
\vcenter{\hbox{$#2#3$ }}\kern-.6\wd0}}

\title[Dimension reduction in optimization problem with measure constraint]{On the dimension of the singular set\\ in optimization problems with measure constraint} 

\author[D.~Mazzoleni, G.~Tortone, B.~Velichkov]{Dario Mazzoleni, Giorgio Tortone, Bozhidar Velichkov }

\address {Dario Mazzoleni \newline \indent
Dipartimento di Matematica, University of Pavia \newline \indent
Via Ferrata 5, 27100 Pavia, Italy}
\email{dario.mazzoleni@unipv.it}

\address {Giorgio Tortone \newline \indent
	Dipartimento di Matematica, Universit\`a di Pisa \newline \indent
	Largo Bruno Pontecorvo, 5, 56127 Pisa, Italy}
\email{giorgio.tortone@dm.unipi.it}

\address {Bozhidar Velichkov \newline \indent
Dipartimento di Matematica, Universit\`a di Pisa \newline \indent
Largo Bruno Pontecorvo, 5, 56127 Pisa, Italy}
\email{bozhidar.velichkov@unipi.it}

\begin{document}

\thanks{{\bf Acknowledgments.}
	G.T. and B.V. are supported by the European Research Council's (ERC) project n.853404 ERC VaReg - \it Variational approach to the regularity of the free boundaries\rm, financed by the program Horizon 2020.
	D.M. and G.T. are members of INDAM-GNAMPA.
	D.M. has been partially supported by the MIUR-PRIN Grant 2020F3NCPX
``Mathematics for industry 4.0 (Math4I4)''.
D.M. and B.V. have been partially supported by the MIUR-PRIN Grant 2022R537CS ``$NO^3$''.}


\begin{abstract}
In this paper, we prove estimates on the dimension of the singular part of the free boundary for solutions to shape optimization problems with measure constraints. The focus is on the heat conduction problem studied by Aguilera, Caffarelli, and Spruck and the one-phase Bernoulli problem with measure constraint introduced by Aguilera, Alt and Caffarelli. To estimate the Hausdorff dimension of the singular set, we introduce a new formulation of the notion of stability for the one-phase problem along volume-preserving variations, which is preserved under blow-up limits. Finally, the result follows by applying the program developed in \cite{BMMTV} to this class of domain variation.
\end{abstract}

\maketitle
\setcounter{tocdepth}{1}


{\centering\footnotesize
  \it Dedicated to Giuseppe Buttazzo, with admiration and gratitude for his mentorship and friendship.\\
  \par}


\section{Introduction}\label{sintro}
In this paper, we prove estimates on the dimension of the singular part of the solutions to free boundary problems with measure constraint.
Although in many physical or engineering problems involving free boundaries a measure constraint naturally arises, the mathematical theory of regularity for these problems has focused mostly on functionals made of a Dirichlet type energy summed to a measure term (see for example~\cite{altcaf}), and thus with no measure constraint.

Here we focus on improving two regularity results for free boundary problems with measure constraints. Precisely, we consider the one-phase Bernoulli problem with measure constraint introduced by Aguilera, Alt and Caffarelli in \cite{agalca} (see \cref{sub:intro-agalca}) and the heat conduction problem of Aguilera, Caffarelli and Spruck \cite{agcasp} (see \cref{sub:intro-agcasp}). In both cases we prove that the full regularity of the free boundary can be pushed up to a critical dimension $d^\ast\in \{5,6,7\}$ (see \cref{sub:intro-critical-dimension} and the discussion in \cite[Section 7]{BMMTV}), while, for $d\ge d^\ast$, we give an estimate on the dimension of the singular set via a dimension reduction principle. In particular, we show that in dimensions $d=2$, $d=3$ and $d=4$, the free boundaries are smooth.

\subsection{The critical dimension $d^\ast$}\label{sub:intro-critical-dimension}
For any open set $D\subset \R^d$ and every function $u\in H^1(D)$, let $J$ be the one-phase Alt-Caffarelli functional
$$J(u,D)=\int_D|\nabla u|^2\,dx+|\{u>0\}\cap D|.$$
We say that a non-negative $u:D\to\R$ is a minimizer of the one-phase functional in $D$ if
$$J(u,D)\le J(v,D)\quad\text{for every}\quad v\in H^1(D)\quad\text{such that}\quad v-u\in H^1_0(D),$$
and we say that a non-negative $u:\R^d\to\R$ is a global minimizer of the one-phase functional if
$$J(u,D)\le J(v,D)\quad\text{for every open}\quad D\subset\R^d\quad\text{and every}\quad v\in H^1(D)\quad\text{with}\quad v-u\in H^1_0(D).$$

\noindent In \cite{weiss}, Weiss showed that if $u$ is a minimizer of the one-phase functional in an open set $D\subset\R^d$, then any blow-up limit of $u$ is a 1-homogeneous global minimizer. This allowed to apply the Federer's dimension reduction principle to minimizers of the one-phase functional showing that:
\begin{itemize}
\item if $d<d^\ast_{min}$, then any minimizer $u$ of $J$ in $D\subset \R^d$ has $C^{1,\alpha}$ free boundary;
\item if $d=d^\ast_{min}$, then a minimizer $u$ of $J$ in $D\subset \R^d$ might have singularities, but they have to be isolated points;
\item if $d>d^\ast_{min}$, then the set of singular free boundary points of a minimizer $u$ of $J$ (in any open $D\subset \R^d$) must have Hausdorff dimension at most $d-d^\ast_{min}$,
\end{itemize}
where $d^\ast_{min}$ is the smallest dimension $d$ in which there is a 1-homogeneous global minimizer $u:\R^d\to\R$ of the one-phase functional with a singularity in zero.
In view of the above results of Weiss, the estimate on the dimension of the singular set of minimizers to the one-phase problem is reduced to finding the critical dimension $d^\ast_{min}$. \medskip

In \cite{cjk}, Caffarelli, Jerison and Kenig used a geometric argument to show that if $u:\R^d\to\R$ is a global 1-homogeneous minimizer with free boundary $\partial\{u>0\}$, which is smooth outside of the origin, then on the positivity set $\Omega_u:=\{u>0\}$  the following \emph{stability inequality} holds:
$$\int_{\Omega_u}|\nabla \varphi|^2\,dx\ge \int_{\partial\Omega_u}|H|\varphi^2\quad\text{for every}\quad \varphi\in C^\infty_c(\R^d\setminus\{0\}),$$
where $H$ is the mean curvature of the boundary $\partial\Omega_u$. Then, they exploited this information to show that $d^\ast_{min}\ge 4$. Later, De Silva and Jerison~\cite{dsj} found an example of a $1$-homogeneous global minimizer with singularity in zero in dimension $7$, thus showing that $d^\ast_{min}\le 7$. Finally, Jerison and Savin~\cite{js} studied functions $u:\R^d\to\R$ satisfying
\begin{itemize}
\item $u$ is non-negative, Lipschitz and 1-homogeneous;
\item $\Delta u=0$ in the open set $\Omega_u:=\{u>0\}$;
\item $\partial\Omega_u\setminus\{0\}$ is $C^\infty$ smooth and $|\nabla u|\equiv1$ on $\partial\Omega_u\setminus\{0\}$;
\item $\Omega_u$ supports the stability inequality of Caffarelli-Jerison-Kenig.
\end{itemize}
By exploiting an argument similar to the one used by Simons for minimal surfaces, they showed that if $u$ is a function satisfying these conditions in dimension $d\le 4$, then $u$ is (the positive part of) a linear function. Thus, they showed that there are no singular minimizers of the one-phase functional in dimension $d\le 4$, which means that $d^\ast_{min}\ge 5$.\medskip

In \cite{BMMTV}, with Buttazzo and Maiale, we studied the regularity of the free boundary of a shape optimization problem for which the analysis cannot be reduced to the study of minimizers or almost-minimizers of the one-phase functional. Precisely, the optimality condition in \cite{BMMTV} could only be expressed in terms of interior variations with smooth vector fields, which made impossible the application of the dimension estimates of Weiss. Thus, we introduced the notion of \emph{global stable solution} of the one-phase Bernoulli problem (see  \cref{def:global-stable-solutions} for the precise definition), which allowed to prove a dimension reduction principle. We also gave the following definition of critical dimension $d^\ast$ for global stable solutions: \emph{$d^\ast$ is the smallest dimension $d$ in which there is a $1$-homogeneous global stable solution $u:\R^d\to\R$ with isolated singularity in zero.} Precisely:
\begin{itemize}
\item If $u$ is a $1$-homogeneous global stable solution in $\R^d$, with $d<d^\ast$, then necessarily
\begin{equation}\label{e:half-plane}
u(x)=(x\cdot\nu)_+\quad\text{for some unit vector}\quad\nu\in\R^d.
\end{equation}
\item In $\R^{d^\ast}$ there is a $1$-homogeneous global stable solution $u$ which is not of the form \eqref{e:half-plane}; the free boundary $\partial\{u>0\}$ is a smooth manifold in $\R^{d^\ast}\setminus\{0\}$ and has a singularity in $0$.
\end{itemize}
Again in \cite{BMMTV}, we showed that the positivity set $\Omega_u$ of a global stable solution $u$ supports the Caffarelli-Jerison-Kenig's stability inequality. Thus, by applying the result of Jerison and Savin, we were able to conclude that $5\le d^\ast\le 7$.\medskip

In the present paper, we use the theory from \cite{BMMTV} to obtain dimension bounds on the singular part of the free boundary for solutions to shape optimization problems with measure constraints. The focus is on the problems studied by Aguilera-Caffarelli-Spruck and Aguilera-Alt-Caffarelli, which we briefly discuss in the subsections \ref{sub:intro-agcasp} and \ref{sub:intro-agalca} below. We notice that, just as in \cite{BMMTV}, the optimization problem of Aguilera-Caffarelli-Spruck cannot be written as a variational minimization problem of a single integral functional. Here, the additional difficulty with respect to \cite{BMMTV} is the presence of a measure constraint, which reduces the class of admissible perturbations.

\subsection{Optimization problem in heat conduction}\label{sub:intro-agcasp}
Let $D$ be a smooth bounded open set in $\R^d$ and let
$$\phi:\partial D\to\R \qquad\text{and}\qquad \varphi:\partial D\to\R$$
be smooth functions on $\partial D$, which are bounded from below and above by positive constants $0<c\le C$. Let $\Omega\subset D$ be an open set, for which the set $K:= D\setminus \Omega$ is a compact subset of $D$. For every such $\Omega$, we denote by  $u_\Omega\in H^1(D)$ the solution to the problem
\begin{equation}\label{e:heat-conduction-u-Omega}
\Delta u_\Omega=0\quad\text{in}\quad\Omega\,,\qquad u_\Omega=\varphi\quad\text{on}\quad \partial D\,,\qquad u_\Omega\equiv 0\quad\text{on}\quad K=D\setminus\Omega\,,
\end{equation}
and we set
\be\label{e:F2}\mathcal F(\Omega):=\int_{\partial D}\phi\frac{\partial u_\Omega}{\partial \nu}\,d\HH^{d-1}\,\ee
where $\nu$ is the outer unit normal to $\partial D$. In \cite{agcasp}, Aguilera, Caffarelli and Spruck studied the following shape optimization problem
\begin{equation}\label{e:heat-conduction-problem}
\min\Big\{\mathcal F(\Omega)\ :\ \Omega\subset D;\ \Omega\,-\,\text{open};\ |\Omega|=m;\ D\setminus\Omega\,-\,\text{compact subset of\ } D\Big\}\,,
\end{equation}
where $m$ is fixed and such that $0<m<|D|$. They showed that there exists a solution $\Omega$ to \eqref{e:heat-conduction-problem}
whose reduced boundary $\partial^\ast\Omega\cap D$ is $C^\infty$ smooth and that there exists $\Lambda>0$ for which
\begin{equation}\label{e:heat-conduction-optimality-condition}
|\nabla u_\Omega||\nabla v_\Omega|=\Lambda\quad\text{on}\quad\partial^\ast\Omega\cap D\,,
\end{equation}
where $v_\Omega$ is the solution to the PDE
\begin{equation}\label{e:heat-conduction-v-Omega}
\Delta v_\Omega=0\quad\text{in}\quad\Omega\,,\qquad v_\Omega=\phi\quad\text{on}\quad \partial D\,,\qquad v_\Omega\equiv 0\quad\text{on}\quad K=D\setminus\Omega\,.
\end{equation}
Precisely, the regularity theorem of Aguilera, Caffarelli and Spruck is the following:
\begin{theo}[Aguilera-Caffarelli-Spruck \cite{agcasp}]\label{t:ACS}
Let $D\subset \R^d$ be a smooth open and bounded set.	There is a solution $\Omega$ to \eqref{e:heat-conduction-problem} whose free boundary $\partial\Omega\cap D$ can be decomposed as the disjoint union of a regular part $\text{\rm Reg}(\partial\Omega)$ and a singular part
	$$\text{\rm Sing}(\partial\Omega):=(\partial\Omega\cap D)\setminus \text{\rm Reg}(\partial\Omega).$$
	The regular set $\text{\rm Reg}(\partial\Omega)$  is a relatively open subset of $\partial\Omega\cap D$ and (locally) a $C^\infty$ manifold.\\ The remaining singular part $\text{\rm Sing}(\partial\Omega)$ is a closed set such that
	$$\HH^{d-1}\Big(\text{\rm Sing}(\partial\Omega)\Big)=0.$$
\end{theo}	

\begin{remark}
We notice that, we can rewrite the functional \eqref{e:F2} in terms of the state variables $u_\O$ and $v_\O$, as follows:
\be\label{e:F2.2}
\mathcal F(\Omega)=\int_{D}\nabla u_\O \cdot\nabla v_\O\,dx.
\ee
\end{remark}

\subsection{One-phase Bernoulli problem with measure constraint}\label{sub:intro-agalca}

Let $D$ be a smooth bounded open set in $\R^d$ and let $\varphi: D\to\R$ be a non-negative function in $H^1(D)$. In \cite{agalca}, Aguilera, Alt and Caffarelli studied the following variational problem
\begin{equation}\label{e:one-phase-constrained}
\min\Big\{\int_D |\nabla u|^2\,dx\\ :\ u\in H^1(D),\  u-\varphi\in H^1_0(D),\ |\{u>0\}\cap D|=m\Big\}\,,
\end{equation}
for any fixed $m>0$ with $0<m<|D|$, and they proved the following theorem.
\begin{theo}[Aguilera-Alt-Caffarelli \cite{agalca}]\label{t:AAC}
Let $D\subset \R^d$ be a smooth open and bounded set. There is a solution $u$ to \eqref{e:one-phase-constrained}. For every solution $u$ with positivity set $\Omega_u:=\{u>0\}$ its free boundary $\partial\Omega_u\cap D$ can be decomposed as the disjoint union of a regular part $\text{\rm Reg}(\partial\Omega_u)$ and a singular part $\text{\rm Sing}(\partial\Omega_u)$.
	The regular part $\text{\rm Reg}(\partial\Omega_u)$  is a relatively open subset of $\partial\Omega_u\cap D$ and (locally) a $C^\infty$ manifold, while the singular is such that:
	$$\HH^{d-1}\Big(\text{\rm Sing}(\partial\Omega_u)\Big)=0.$$
\end{theo}	

\begin{remark}
In \cite{agalca} it was shown that there exists a minimizer which is regular; for the regularity of every minimizer, as stated in the above theorem, we refer to \cite[Chapter 10]{V}, where it is used an argument inspired by~\cite{brianconlamboley} and~\cite{briancon}.
\end{remark}

\begin{remark}\label{oss:agcasp-implies-agalca}
We notice that it is possible to write \eqref{e:one-phase-constrained} as a shape optimization problem in terms of the shape functional
\be\label{e:F1}
\mathcal F(\Omega)=\int_{D}|\nabla u_\O|^2\,dx,
\ee
where for every $\O\subset D$, the state function $u_\O\in H^1(D)$ is the solution to the PDE
\begin{equation}\label{e:state-function-aac-intro}
\Delta u_\Omega=0\quad\text{in}\quad\Omega\,,\qquad u_\Omega-\varphi\in H^1_0(\Omega).
\end{equation}
Precisely, if $u$ minimizes \eqref{e:one-phase-constrained}, then the set $\Omega:=\{u>0\}$ is open and
$$\mathcal F(\Omega)\le \mathcal F(\widetilde \Omega),\qquad\mbox{for every open set }\widetilde\Omega\subset D\quad\mbox{with}\quad |\widetilde\Omega|=m.$$
Conversely, if $\Omega$ is an open subset of $D$, minimizing $\mathcal F$ among all open subsets of fixed measure $m$, then the solution $u_\Omega$ to \eqref{e:state-function-aac-intro} is a solution to \eqref{e:one-phase-constrained}. Thus, problem \eqref{e:one-phase-constrained} can be seen as a particular case of \eqref{e:heat-conduction-problem} in which $\psi=\varphi$ and $u_\Omega=v_\Omega$.
\end{remark}

\begin{remark}\label{rem.saturated}
We also notice that both problems \eqref{e:heat-conduction-problem} and \eqref{e:one-phase-constrained} are respectively equivalent to the auxiliary problems
  $$ \min\Big\{\mathcal F(\Omega)\ :\ \Omega\subset D;\ \Omega\,-\,\text{open};\ |\Omega|\leq m\Big\}.$$
Indeed, it is not hard to show that there is a minimizer to this auxiliary problem and that every minimizer saturates the measure constraint; we refer to \cite[Proposition 11.1]{V} for \eqref{e:one-phase-constrained} and \cite[Theorem 1.5]{BMV1} for problem \eqref{e:heat-conduction-problem}.
\end{remark}

\subsection{Main results} In this paper we enhance the results of Aguilera-Caffarelli-Spruck \cite{agcasp} and Aguilera-Alt-Caffarelli \cite{agalca} by improving the estimates on the dimension of the singular sets on the free boundaries of the solutions to \eqref{e:one-phase-constrained} and \eqref{e:heat-conduction-problem}. In particular, we show that in the physically relevant dimension $d=3$ (and also in $d=2$ and $d=4$) the free boundaries are $C^\infty$ smooth. Our main results are the following.

\begin{theo}[Dimension of the singular set for solutions to \eqref{e:heat-conduction-problem}]\label{t:heat-conduction-singular}
	Let $\Omega$ be as in the statement of \cref{t:ACS}, then the following  holds:	
	\begin{enumerate}[\rm (i)]
		\item if $d< d^\ast$, then $\text{\rm Sing}(\partial\Omega)=\emptyset$;
		\item if $d\ge d^\ast$, then the Hausdorff dimension of $\text{\rm Sing}(\partial\Omega)$ is at most $d-d^\ast$;
	\end{enumerate}		
	where $d^\ast$ is the critical dimension from \cref{sub:intro-critical-dimension}.
\end{theo}	
Similarly, the vectorial problem \eqref{e:heat-conduction-problem} immediately implies the following result for its scalar counterpart, that is the one-phase problem with measure constraint \eqref{e:one-phase-constrained}.
\begin{theo}[Dimension of the singular set for solutions to \eqref{e:one-phase-constrained}]\label{t:one-phase-singular}
	Let $u$ be as in the statement of \cref{t:AAC}. Then the following holds:	
	\begin{enumerate}[\rm (i)]
		\item if $d< d^\ast$, then $\text{\rm Sing}(\partial\Omega_u)=\emptyset$;
		\item if $d\ge d^\ast$, then the Hausdorff dimension of $\text{\rm Sing}(\partial\Omega_u)$ is at most $d-d^\ast$;
	\end{enumerate}		
	where $d^\ast$ is the critical dimension from \cref{sub:intro-critical-dimension}.
\end{theo}	


As explained in \cref{oss:agcasp-implies-agalca}, \cref{t:one-phase-singular} follows from  \cref{t:heat-conduction-singular}. Thus, we will focus on the heat conduction problem of Aguilera-Caffarelli-Spruck. From now on, given $D$ a smooth bounded open set in $\R^d$, we set $\O$ to be a minimizer of
\begin{equation*}
\min\Big\{\mathcal F(\Omega)\ :\ \Omega\subset D;\ \Omega\,-\,\text{open};\ |\Omega|=m\Big\}\,,
\end{equation*}
where $0<m<|D|$ and where $\mathcal{F}$ is given by \eqref{e:F2.2}. \medskip


The main steps in the proof of \cref{t:heat-conduction-singular} are the following.\medskip

First, given any smooth vector field with compact support $\eta \in C^\infty_c(\R^d;\R^d)$, we construct a volume preserving flow, whose infinitesimal generator coincides with $\eta$ in a neighborhood of $\{\eta\neq 0\}$. Then, we show that if $\O$ is an optimal set for \eqref{e:heat-conduction-problem}, then there exists a Lagrange multiplier $\Lambda \geq 0$ such that $\Omega$  is stable for the functional
    $$
\mathcal{F}_\Lambda(\O) = \mathcal{F}(\O)  + \Lambda |\O|,
$$
with respect to volume-preserving domain variations (see \cref{t:var}).
Moreover, by exploiting the PDEs satisfied by the state functions associated to the optimal shape, we prove that $\Lambda >0$. For this first part of the proof, we follow a general well-known program for measure-constrained optimization problems (see, for instance, \cite[Section 11]{V} and \cite[Section 17.5]{Maggi}).\medskip

In the second part of the proof, we show that the estimates on the dimension of the singular part hold for solutions which are stable with respect to volume-preserving variations. The key result is the following.
\begin{theo}\label{t:AC}
Suppose that $u:\R^d\to\R$ is a $1$-homogeneous global stable solution of the one-phase problem along mean-zero variations (in the sense of \cref{def:global-stable-solutions-meanzero}). Then, if $\O_u = \{u>0\}$, the following holds:
\begin{enumerate}[\rm (i)]
\item if $d< d^\ast$, then $u$ is a half-plane solution\,\rm;\it
\item if $d=d^\ast$, then $\text{\rm Sing}(\partial\O_u)=\{0\}$ or $u$ is a half-plane solution\,\rm;\it
\item if $d>d^\ast$, then the Hausdorff dimension of $\text{\rm Sing}(\partial\O_u)$ is at most $d-d^\ast$, that is,
$$\HH^{d-d^\ast+s}\big(\text{\rm Sing}(\partial\O_u)\big)=0\quad\text{for every}\quad s>0\,\rm;\it$$
\end{enumerate}		
where $d^\ast$ is the critical dimension from \cref{def:global-stable-solution-critical-dimension}, satisfying $5\leq d^\ast\leq 7$.
\end{theo}

\noindent Finally, using this result, together with the Federer's dimension reduction principle and the blow-up analysis carried out in \cite{BMMTV}, we conclude that in dimension $d< d^*$ all the points on the optimal free boundary $\partial\O\cap D$ are regular, while in higher dimension, we prove the Hausdorff estimates from \cref{t:heat-conduction-singular}.


\subsection{Plan of the paper}
In \cref{s:volume-contrained-variation} we show the existence of volume-preserving inner domain variations and we prove the validity of a stability-type condition for minimizers of $\mathcal{F}$ along a suitable class of volume-preserving flows. In \cref{s:AC} we recall the notion of a global stable solution to the one-phase Bernoulli  problem introduced in \cite{BMMTV} and we study the case of stability with respect to volume-preserving variations. This section can also be read independently and, together with \cref{s:volume-contrained-variation}, contains the key results for the analysis of the singular set for volume-constrained free boundary problems. In \cref{s:agcafs} we apply the previous analysis to the heat conduction problem \eqref{e:heat-conduction-problem}. More precisely, in light of the blow-up analysis developed in \cite{BMMTV}, we connect the occurrence of singularity at a macroscopic level to the possible existence of singular stable cones for the Alt-Caffarelli functional, in the sense of \cref{s:AC}. Ultimately, the estimates in \cref{t:heat-conduction-singular} follow by Federer's reduction principle.

\subsection{Notations}\label{s:notations}
Through the paper, we use the notation
$$B_r(x_0)=\{x\in \R^d : |x-x_0|<r\},$$
for the ball of radius $r$ centered at point $x_0$ and, when $x_0=0$, we simply write $B_r=B_r(0)$. For any set $E\subset \R^d$, we define
$$\mathds{1}_E(x):=\begin{cases}
1, &\text{if }\ x\in E,\\
0, &\text{if }\ x\in \R^d\setminus E.
\end{cases}$$
Given a non-negative function $u$ we often denote its positivity set as $\O_u:=\{u>0\}$.

We denote by $H^1(\R^d)$ the set of Sobolev functions in $\R^d$, that is, the closure of the smooth functions with bounded support in $\R^d$ with respect to the usual Sobolev norm
$$\|\phi\|_{H^1}^2=\int_{\R^d}\big(|\nabla\phi|^2+\phi^2\big)\,dx\,.$$
Similarly, for an open set $\O\subset\R^d$, we define the Sobolev space $H^1_0(\O)$ as the closure of $C^\infty_c(\O)$ with respect to $\|\cdot\|_{H^1}$. Therefore, every $u\in H^1_0(\O)$ is identically zero outside the set $\O$ and the inclusion $H^1_0(\O)\subset H^1(\R^d)$ holds true. When dealing with PDEs on cones we will use the space \[
\dot H^1(\R^d)=\{u\in L^{2^*}(\R^d) : \nabla u\in L^2(\R^d,\R^d)\},
\]
equipped with the usual norm $\|u\|_{\dot H^1}=\|\nabla u\|_{L^2}$, and $\dot H^1_0(\Omega)$ is the closure of $C^\infty_c(\Omega)$ in $\dot H^1(\R^d)$. If $\Omega$ satisfies uniform exterior density estimates, the following characterization holds: \[
\dot H^1_0(\Omega)=\{u\in \dot H^1(\R^d) : u=0\text{ a.e. in }\R^d\setminus \Omega\}.
\]
For any function $u:\R^d\to\R$ we will denote by $D u$ and $\nabla u$ the vectors row and column with components the partial derivatives of $u$, while $D^2u$ will be the Hessian matrix of $u$.

\section{Volume-preserving variations}\label{s:volume-contrained-variation}
In this section we define a suitable collection of volume-preserving  variations of a domain $\Omega$.

Let $D\subset \R^d$ be a bounded set and $\eta\in C^\infty_c(D;\R^d)$ be a smooth compactly supported vector field in $D$. Let $\Phi_t:D\to D$, $t\in\R$ be the flow associated to $\eta$ and consider the sets $\Omega_t:=\Phi_t(\Omega)$. Then, by direct computations
\be\label{e:expression}\begin{aligned}
\delta |\O|[\eta] &:= \frac{\partial}{\partial t}\bigg|_{t=0}|\O_t|=\int_\O\dive\eta\,dx\\
\delta^2 |\O|[\eta] &:= \frac{\partial^2}{\partial t^2}\bigg|_{t=0}|\O_t|=
\int_\O (\dive\eta)^2+\eta\cdot\nabla(\dive\eta)\,dx.
\end{aligned}\ee
Similarly, in view of the analysis conducted in \cite[Section 2.1]{BMMTV}, it is possible to compute the second variation of the functional $\mathcal{F}$ along internal smooth variations with compact supports. Indeed, by recalling the definition of $\mathcal{F}$ in \eqref{e:F2.2}, we have
\begin{align}\label{e:deriv-F}
\begin{aligned}
\delta\mathcal{F}(\O)[\eta] &
=\, \int_\O \nabla u_\Omega \cdot (\delta A)\nabla v_\Omega  \, dx\,,\\
\delta^2 \mathcal{F}(\O)[\eta] &
=\,
2\int_\O\Big(\nabla u_\Omega\cdot(\delta^2 A)\nabla v_\Omega-\nabla (\delta u)\cdot \nabla (\delta v)\Big) dx\,,
\end{aligned}
\end{align}
where the matrices $\delta A\in L^\infty(D;\R^{d\times d})$ and $\delta^2A\in L^\infty(D;\R^{d\times d})$ are given by
	\begin{align}\label{e:first-and-second-variation-A}
	\begin{aligned}
	\delta A &:=-D\eta-\nabla\eta+(\dive\eta)\text{\rm Id}\,,\\
	\delta^2 A &:= (D\eta)\,(\nabla\eta)+\frac12\big(\nabla\eta\big)^2+\frac12(D\eta)^2-\frac12(\eta\cdot\nabla)\big[\nabla\eta+D\eta\big]\\
	&\qquad-\big(\nabla\eta+D\eta\big)\dive\eta+\text{\rm Id}\frac{(\dive\eta)^2+\eta\cdot\nabla(\dive\eta)}{2}\,,
	\end{aligned}
	\end{align}
and where $\delta u$ and $\delta v$ are the solutions to the PDEs
\be\label{e:deltau}
\begin{aligned}
-\Delta (\delta u)=\dive\big((\delta A)\nabla u_\Omega\big)\quad\text{in}\quad \Omega\ ,\qquad \delta u\in  H^1_0(\Omega)\,,\\
-\Delta (\delta v)=\dive\big((\delta A)\nabla v_\Omega\big)\quad\text{in}\quad \Omega\ ,\qquad \delta v\in H^1_0(\Omega)\,.
\end{aligned}
\ee
Clearly, in the case of the functional $\mathcal{F}$, in which the dependence on the shape $\O$ is formulated in terms the state variables $u_\O, v_\O$ (e.g. a single state variable $u_\O$ for \eqref{e:one-phase-constrained}), the previous derivatives must be calculated taking into account the effect of the variation of the domains on the new state variable supported on the perturbed domain. For more details on the derivation of \eqref{e:deriv-F} we refer to the analysis carried out in \cite[Section 2]{BMMTV}.\medskip

The following is the main result of this section.
\begin{theo}\label{t:var}
Let $D$ be a bounded open set in $\R^d$ and $\Omega \subseteq D$ be a minimizer to \eqref{e:heat-conduction-problem}. Then, there exists $\Lambda \geq 0$ such that, if
$$\mathcal{F}_\Lambda(\O) = \mathcal{F}(\O)  + \Lambda |\O|\,,$$
then
\begin{equation}
\delta \mathcal F_\Lambda(\O)[\eta]=0\qquad\text{and}\qquad \delta^2 \mathcal F_\Lambda(\O)\left[\eta - \frac{\int_\O \dive\eta\,dx}{\int_\O \dive\xi\,dx}\xi\right]\ge0\,,
\end{equation}
for every couple $\eta,\xi \in C^\infty_c(\R^d;\R^d)$ of smooth vector fields with disjoint compact supports satisfying
$$
\int_\O \dive\xi \,dx\neq 0.
$$
\end{theo}
\begin{remark}\label{r:deltaA}
The second variation in \eqref{e:deriv-F} can be calculated by solving the system \eqref{e:deltau}, with $\delta A$ and $\delta^2 A$ defined starting from the field
$$x \mapsto \eta(x) - \frac{\int_\O \dive\eta\,dx}{\int_\O \dive\xi\,dx}\,\xi(x).$$
\end{remark}
Before proceeding with the proof, we start by proving the existence of a volume-preserving variation for every open shape.
\begin{lemma}[Volume-preserving variations]\label{l1}

Let $D$ be a bounded open set in $\R^d, \Omega \subseteq D$ be an open set and $\eta \in C^\infty_c(D;\R^d)$ be a compactly supported vector field.\\ Then, for every vector field $\xi \in C^\infty_c(D;\R^d)$ satisfying $$\delta |\O|[\xi] = \int_\O \dive\xi \,dx \neq 0\quad\mbox{and}\quad\text{\rm supp}\,\eta \cap \text{\rm supp}\,\xi=\emptyset,$$
there exists $t_0>0$ and $f\in C^\infty((-t_0,t_0))$ such that the map
$$\Phi_t(x)=\text{\rm Id}(x) + t\eta(x)+f(t)\xi(x)$$
preserves the volume of $\O$, for every $t \in (-t_0,t_0)$, that is
$$|\Phi_t(\O)|=|\O|\qquad\text{for every}\qquad t \in (-t_0,t_0).$$
Moreover, the following identities hold true
\be\label{es}
\begin{aligned}
&f(0)=0,\quad  f'(0) \int_\O \dive\xi\,dx= - \int_\O \dive\eta\,dx\\
&f''(0) \int_\O \dive\xi\,dx= - \frac12
\int_\O (\dive\eta)^2+\eta\cdot\nabla(\dive\eta)  + (f'(0))^2 (\dive\xi)^2+(f'(0))^2\xi\cdot\nabla(\dive\xi)\,dx.
\end{aligned}
\ee
\end{lemma}
\begin{proof}
  The proof is a straightforward application of the implicit function Theorem to the map
  $$
  \varphi(t,s)= |(\text{\rm Id}+ t \eta + s \xi)(\O)|- |\O|,\quad \mbox{such that}\quad \varphi(0,0)=0.
  $$
  Indeed, since
  $$
  \partial_s \varphi(0,0) = \int_\O \dive\xi \, dx \neq 0,
  $$
  there exists $t_0>0$, sufficiently small, and $f\in C^\infty((-t_0,t_0))$ such that
  $$
  \varphi(t,f(t))=0\qquad\mbox{for every}\qquad t \in (-t_0,t_0).
  $$
Finally, the estimates in \eqref{es} follow by a direct computation. Indeed, since
$$
0=\frac{\partial}{\partial t}\bigg|_{t=0}\varphi(t,f(t))= \partial_t\varphi(0,0) + f'(0)\partial_s \varphi(0,0)$$
we have that
$$
f'(0)= -\frac{\partial_t\varphi(0,0)}{\partial_s \varphi(0,0)} = -\frac{\delta |\O|[\eta]}{\delta |\O|[\xi]}.
$$
Similarly,
$$
0=\frac{\partial^2}{\partial t^2}\bigg|_{t=0}\varphi(t,f(t))= \partial^2_t\varphi(0,0) + 2 \partial_t\partial_s \varphi(0,0) f'(0) + f''(0)\partial_s \varphi(0,0)+ (f'(0))^2\partial^2_s\varphi(0,0)$$
which implies
\begin{equation}\label{e:f''}
\begin{split}
f''(0) &= -\frac{1}{\partial_2 \varphi(0,0)}\left(
\partial^2_t\varphi(0,0) + 2 \partial_t\partial_s \varphi(0,0) f'(0) + (f'(0))^2\partial^2_s\varphi(0,0)\right)\\
&= -\frac{1}{\delta |\O|[\xi]}\left(
\delta^2 |\O|[\eta] + 2 f'(0) \delta^2 |\O|[\eta,\xi] + (f'(0))^2 \delta^2 |\O|[\xi]\right)\\
&= -\frac{1}{\delta |\O|[\xi]}\left(
\delta^2 |\O|[\eta] + (f'(0))^2 \delta^2 |\O|[\xi]\right)
\end{split}
\end{equation}
since $\text{\rm supp}\,\eta \cap \text{\rm supp}\,\xi=\emptyset$. The claim now follows by \eqref{e:expression}.
\end{proof}
The proof of \cref{t:var} is now a consequence of the following propositions. In \cref{p:first}, we show the existence of a Lagrange multiplier along inner variations. Then, in \cref{p:second}, we prove the validity of a stability-type condition with respect to internal variations.
\begin{prop}[First variation]\label{p:first}
Let $D$ be a bounded open set in $\R^d$ and $\Omega \subseteq D$ be a minimizer to \eqref{e:heat-conduction-problem}. Then, there exists $\Lambda \geq 0$ such that
  $$
  \delta\mathcal{F}(\O)[\eta ] + \Lambda \delta|\O|[\eta]=0,
  $$
  for every $\eta \in C^\infty_c(D;\R^d)$.
\end{prop}
\begin{proof}
Let $\eta,\xi\in C^\infty_c(D;\R^d)$ be two smooth vector field with disjoint compact supports, satisfying
$$
\textrm{supp}\,\eta \cap \O \neq \emptyset\qquad\mbox{and}\qquad \textrm{supp}\,\xi \cap \O \neq \emptyset.
$$
Then, for $t_0>0$ sufficiently small we consider the volume-preserving flow from \cref{l1}
$$\Phi \colon (-t_0,t_0)\times \R^d\to \R^d\ ,\qquad \Phi_t(x) = \text{\rm Id}(x) + t \eta(x) + f(t) \xi(x),$$
with $f$ satisfying \eqref{es}. Now, since $|\Phi_t(\O)|=|\O|$ for every $t$, we have
$$
0=\frac{d}{dt}\bigg|_{t=0}|\Phi_t(\O)|= \delta |\O|[\eta + f'(0)\xi] = \delta|\O|[\eta] + f'(0) \delta |\O|[\xi],
$$
and, by the minimality of $\Omega$ among the sets of the same volume, we have
$$
\frac{d}{dt}\bigg|_{t=0}\mathcal{F}(\Phi_t(\O))=0.
$$
Finally, by exploiting the linearity of the variation of $\mathcal{F}$ with respect to smooth vector field with compact supports, we infer that
$$
0=\frac{d}{dt}\bigg|_{t=0}\mathcal{F}(\Phi_t(\O)) = \delta \mathcal{F}(\O)[\eta] + f'(0)\delta\mathcal{F}(\O)[\xi] = \delta \mathcal{F}(\O)[\eta] -\left(\frac{\delta|\O|[\eta]}{\delta |\O|[\xi]}\right)\delta\mathcal{F}(\O)[\xi]
$$
which implies that
$$
\frac{\delta \mathcal{F}(\O)[\eta]}{\delta|\O|[\eta]} = \frac{\delta \mathcal{F}(\O)[\xi]}{\delta |\O|[\xi]},
$$
for every $\eta, \xi \in C^\infty_c(D;\R^d)$. Now, by \cite[Lemma 11.3]{V}, there is $\eta_0 \in C^\infty_c(D;\R^d)$ such that
$$\delta |\O|[\eta_0] = \int_\O \dive\eta_0\,dx = 1.$$
Then choosing $\xi=\eta_0$, for every $\eta \in C^\infty_c(D;\R^d)$ we have
$$
\delta \mathcal{F}(\O)[\eta]+\Lambda \delta|\O|[\eta]=0
$$
with $\Lambda=-\delta\mathcal{F}(\O)[\eta_0]$. Ultimately, it is straightforward to notice that $\Lambda\geq 0$. If not, it holds $\delta\mathcal F(\Omega)[\eta_0]>0$, and we consider the flow associate to the vector field $\xi = -\eta_0\in C^\infty_c(D;\R^d)$, where $\eta_0$ is again as in \cite[Lemma 11.3]{V}. Since $\delta |\O|[\xi] = -1$, there exists $t_0>0$ small such that
$$
|\Phi_t(\O)|\leq |\O|,\qquad \mbox{for }t \in [0,t_0).
$$
Therefore, by  \cref{rem.saturated}, we can consider the family $\Phi_t(\O)$ as admissible competitor for $\O$, for $t$ sufficiently small. Finally, the contradiction follows by computing the shape derivative along the flow generated by $\xi$, that is
$$
\delta\mathcal{F}(\O)[\xi] = -\delta \mathcal{F}(\O)[\eta_0] = \Lambda\geq 0,
$$
which contradicts the minimality of the set $\O$.
\end{proof}
\begin{prop}[Second variation]\label{p:second}
Let $D$ be a bounded open set in $\R^d$ and $\Omega \subseteq D$ be a minimizer to \eqref{e:heat-conduction-problem}. Given $\Lambda\geq0$ be the constant from \cref{p:first}, then
  $$
  \delta^2 \mathcal{F}(\O)\left[\eta - \frac{\int_\O \dive\eta\,dx}{\int_\O \dive\xi\,dx}\xi\right] + \Lambda \delta^2 |\O|\left[\eta - \frac{\int_\O \dive\eta\,dx}{\int_\O \dive\xi\,dx}\xi\right]\geq 0,
  $$
  for every couple of smooth vector field $\eta,\xi \in C^\infty_c(D;\R^d)$ with disjoint support and satisfying $$
  \delta |\O|[\xi] = \int_\O \dive\xi\,dx \neq 0.
  $$
  \end{prop}
\begin{proof}
  Let $\eta,\xi\in C^\infty_c(D;\R^d)$ be smooth vector fields with disjoint compact supports, satisfying $$
  \textrm{supp}\,\eta \cap \O \neq \emptyset, \qquad \textrm{supp}\,\xi \cap \O \neq \emptyset\quad\mbox{and}\quad \int_\O \dive\xi\,dx\neq 0.
   $$
   Then, consider the volume-preserving flow from \cref{l1}
   $$\Phi \colon (-t_0,t_0)\times \R^d\to \R^d\ ,\qquad \Phi_t(x) = \text{\rm Id}(x) + t \eta(x) + f(t) \xi(x),$$
with $f$ satisfying \eqref{es}. Then, by exploiting the differentiability of the functional $\mathcal{F}$ along $\Phi_t$ (see \cite[Proposition 2.8]{BMMTV}), we infer that
\begin{align*}
\mathcal{F}(\Phi_t(\O)) =\, \mathcal{F}(\O) &+ t\,\delta \mathcal{F}(\O)\big[\eta + f'(0)\xi\big]\\
&+ \frac12t^2\Big(\delta^2 \mathcal{F}(\O)\big[\eta + f'(0)\xi\big] + \delta \mathcal{F}(\O)\big[f''(0)\xi\big]\Big)+ o(t^2).
\end{align*}
First we note that, using~\eqref{e:f''},
$$
\delta\mathcal{F}(\O)[f''(0)\xi] = -\frac{\delta \mathcal{F}(\O)[\xi]}{\delta |\O|[\xi]}\left(\delta^2 |\O|[\eta] + (f'(0))^2 \delta^2 |\O|[\xi]\right) = \Lambda\left(\delta^2 |\O|[\eta] + (f'(0))^2 \delta^2|\O|[\xi]\right),
$$
which, together with \cref{p:first} and the minimality of $\Omega$, implies that
$$
\delta^2\mathcal{F}(\O)[\eta+f'(0)\xi] + \Lambda\left(\delta^2 |\O|[\eta] +(f'(0))^2 \delta^2|\O|[\xi]\right) \geq 0.
$$
Ultimately, since $\text{\rm supp}\,\eta \cap \text{\rm supp}\,\xi = \emptyset$, we can rewrite the last two terms in the left-hand side as
$$
\delta^2\mathcal{F}(\O)[\eta+f'(0)\xi] + \Lambda \delta^2 |\O|[\eta +f'(0)\xi]\geq 0.
$$
The claim now follows by the first estimate in \eqref{es}.
\end{proof}
\begin{remark}
  In this paper, the focus is on the volume-preserving first and second variations to the functionals introduced in \eqref{e:heat-conduction-problem} and \eqref{e:one-phase-constrained}. Nonetheless, the results of this section are still valid under some mild assumptions on the functionals as the differentiability of the state functions with respect to smooth internal variations with compact support (in the sense of \cite[Section 2.1]{BMMTV}).
\end{remark}
\section{Stable homogeneous solutions of the one-phase Bernoulli problem}\label{s:AC}

In this section we study the singular set of the stable global solutions of the one-phase Bernoulli problem (see \cref{def:global-stable-solutions} for the definition of global stable solution) along a suitable class of internal variations. The main results is \cref{t:AC}, which we use in \cref{t:heat-conduction-singular} (resp. \cref{t:one-phase-singular}) in  \cref{s:agcafs}, in order to estimate the dimension of the singular set of the free boundary of the optimal sets for \eqref{e:heat-conduction-problem} (resp. \eqref{e:one-phase-constrained}). The present section is of independent interest and builds up on the theory developed in \cite[Section 7]{BMMTV}.\\

Before going into the proof of \cref{t:AC}, let us recall the notion of global stable solution to the one-phase problem (Alt-Caffarelli) introduced in \cite{BMMTV}. Given a Lipschitz function $u\colon \R^d\to \R$ and $\Lambda>0$ we consider the Alt-Caffarelli functional
$$
\mathcal G_\Lambda(u):=\int_{\R^d}|\nabla u|^2\,dx+\Lambda|\{u>0\}\cap \R^d|,\quad\text{with}\quad \Lambda>0\,.
$$
Then, let $\O_u = \{u>0\}$ and
$$\delta \mathcal G_\Lambda: C^{0,1}(\R)\times C^\infty_c(\R^d;\R^d)\to\R,\qquad\delta^2\mathcal G_\Lambda: C^{0,1}(\R)\times C^\infty_c(\R^d;\R^d)\to\R\,,$$
be the functionals defined as follows: for $\eta \in C^\infty_c(\R^d;\R^d)$ set
\begin{align}\label{e:global-stable-solutions-delta}
\begin{aligned}
\delta \mathcal G_\Lambda(u)[\eta]&:=\int_{\O_u} \Big(\nabla u \cdot \delta A\nabla u+ \Lambda \dive\eta\Big)\, dx\,,\\
\delta^2 \mathcal G_\Lambda(u)[\eta]&:=\,
\int_{\O_u}2\nabla u\cdot(\delta^2 A)\nabla u-2|\nabla (\delta u)|^2 +\Lambda\Big((\dive\eta)^2+\eta\cdot\nabla(\dive\eta)\Big)\,dx,
\end{aligned}
\end{align}
where $\delta A,\,\delta^2A$ are defined in \eqref{e:first-and-second-variation-A}
and where $\delta u\in \dot H^1_0(\O_u)$ is the weak solution to
\be\label{e:global-stable-solutions-delta-u}
-\Delta(\delta u)=\dive\big((\delta A)\nabla u\big)\quad\text{in }\O_u\ ,\qquad\delta u\in \dot H^1_0(\O_u).
\ee

\begin{definition}[Global stable solutions of the one-phase problem]\label{def:global-stable-solutions}
A function $u:\R^d\to\R$ is said to be a ``global stable solution of the one-phase problem'' if, for every smooth vector field with compact support $\eta \in C^\infty_c(\R^d;\R^d)$, we have:
\be\label{e:global-stable-solutions-main}
\delta \mathcal G_\Lambda(u)[\eta]=0\qquad\text{and}\qquad \delta^2 \mathcal G_\Lambda(u)[\eta]\ge0\,,
\ee
and if the following conditions are fulfilled:
\begin{enumerate}
\item[$\qquad(\mathcal A1)$]\label{item:def-stable-1} $u$ is globally Lipschitz continuous and non-negative on $\R^d$\rm ;\it
\item[$\qquad(\mathcal A2)$]\label{item:def-stable-2} $u$ is harmonic in the open set $\O_u$\rm;\it
\item[$\qquad(\mathcal A3)$]\label{item:def-stable-3} there is a constant $c>0$ such that
$$|B_r(x_0)\cap \O_u| \leq (1-c)|B_r|,$$
for every $x_0\in \R^d\setminus\O_u$ and every $r>0$\rm ;\it
\item[$\qquad(\mathcal A4)$]\label{item:def-stable-4} $0\in\partial\O_u$ and there is a constant $C > 0$ such that
$$\sup_{B_r(x_0)} u \ge \frac{1}{C} r\quad\text{for every}\quad x_0\in \overline\O_{u}\,\text{ and }\,r>0\,;$$
\item[$\qquad(\mathcal A5)$]\label{item:def-stable-5} there is a constant $C>0$ such that, for every $R>0$,
$$\left|\int_{B_{R}}\nabla u\cdot\nabla\varphi\,dx\right|\le C R^{d-1}\|\varphi\|_{L^\infty(B_{R})}\quad\text{for every}\quad \varphi\in C^\infty_c(B_R)\,.$$
\end{enumerate}
\end{definition}
\begin{definition}[Critical dimension for stable solutions]\label{def:global-stable-solution-critical-dimension}
We define $d^\ast$ to be the smallest dimension admitting a $1$-homogeneous global stable solution $u:\R^d\to\R$ with $0\in\text{\rm Sing}(\partial\O_u)$.
\end{definition}
In \cite{BMMTV}, with Buttazzo and Maiale, we developed a theory about the regularity of the stable global solutions of the one-phase Alt-Caffarelli problem and we proved the existence of a critical dimension $d^\ast$ (in the sense of \cref{def:global-stable-solution-critical-dimension}), in which a singular global stable solution appears for the first time (see \cite[Theorem 7.8]{BMMTV}), and we show that the $d^\ast \in \{5,6,7\}$ (see \cite[Theorem 7.9]{BMMTV}).
In \cref{t:AC} we extend the previous analysis to the case of stable solutions of the one-phase problem with respect to mean-zero variations.
\begin{definition}[Global stability along mean-zero variations of the one-phase problem]\label{def:global-stable-solutions-meanzero}
A function $u:\R^d\to\R$ is said to be a global stable solution of the one-phase problem ``along mean-zero variations'' if:
\begin{enumerate}
\item[\rm(1)] conditions $(\mathcal A1), (\mathcal A2), (\mathcal A3)$, $(\mathcal A4)$, $(\mathcal A5)$ of \cref{def:global-stable-solutions} are satisfied\,\rm;\it
\item[\rm(2)] for every couple $\eta,\xi \in C^\infty_c(\R^d;\R^d)$ of smooth vector fields with disjoint compact supports satisfying
    $$
\int_{\O_u} \dive\xi \,dx\neq 0,
$$
it holds that
\be\label{e:stab.mean}
\delta \mathcal G_\Lambda(u)[\eta]=0\qquad\text{and}\qquad \delta^2 \mathcal G_\Lambda(u)\left[\eta - \frac{\int_{\O_u} \dive\eta\,dx}{\int_{\O_u} \dive\xi\,dx}\xi\right]\ge0\,.
\ee
\end{enumerate}
\end{definition}
We highlight that the conditions $(1)$-$(2)$ in \cref{def:global-stable-solutions-meanzero} are satisfied by the blow-ups of numerous shape optimization problems, for instance by the state functions of the domains arising from \eqref{e:heat-conduction-problem} (see  \cref{s:agcafs}), and naturally, by the global minimizers of the Alt-Caffarelli functional.
\subsection{Proof of \cref{t:AC}}
The proof consists in showing that the hypotheses from \cref{def:global-stable-solutions-meanzero} imply the stability condition \eqref{e:global-stable-solutions-main}, that is global stable solutions along mean-zero variations are actually stable in the sense of \cref{def:global-stable-solutions}. Then, the claim follows by applying \cite[Theorem 7.8]{BMMTV} and \cite[Theorem 7.9]{BMMTV}.\\

In view of the classical Federer's reduction principle (see for instance \cite[Proof of Proposition 10.13]{V}), it is not restrictive to assume that $\partial\O_u$ has an isolated singularity at the origin, that is
$$
\text{\rm Sing}(\partial\O_u)=\{0\} \qquad \mbox{and} \qquad \text{\rm Reg}(\partial\O_u)=\partial \O_u\setminus\{0\}
$$
and that $u$ is a classical solution to the PDE
\be\label{e:smooth-one-phase}
\Delta u=0\quad\text{in }\O_u\,,\qquad |\nabla u|=\sqrt{\Lambda}\quad\text{on }\text{\rm Reg}(\partial\O_u)\,.
\ee
We stress that, in light of the results from \cite{agcasp} (which we recalled in \cref{t:ACS}), $\text{\rm Reg}(\partial\O_u)$ is locally the graph of a $C^\infty$ smooth function.
Now, given $\eta \in C^\infty_c(\R^d;\R^d)$, we fix
$$R>0\qquad\text{and}\qquad x_0 \in \text{\rm Reg}(\partial \O_u)\cap \partial B_{2R},$$ such that $$\overline{B_R(x_0)}\cap \text{\rm supp}\,\eta = \emptyset\qquad\mbox{and}\qquad 0\not\in \overline{B_R(x_0)}.$$
Then, by \eqref{e:stab.mean} and the fact that $\xi\mapsto \delta^2 \mathcal{G}_\lambda(u)[\xi]$ is quadratic, for every $\xi \in C^\infty_c(B_R(x_0);\R^d)$ with support disjoint from that of $\eta$, we have
\be\label{e:dai}
0 \leq \delta^2 \mathcal G_\Lambda(u)\left[\eta - \xi\int_{\O_u} \dive\eta\,dx\right] = \delta^2 \mathcal G_\Lambda(u)[\eta] + \left(\int_{\O_u} \dive\eta\,dx\right)^2 \delta^2 \mathcal{G}_\lambda(u)[\xi].
\ee
Since $\overline{B_R(x_0)}\cap \partial \O_u = \text{\rm Reg}(\partial \O_u)$, by \cite[Proposition 7.12]{BMMTV} (in particular, Step 4 of its proof), we have
$$
\delta^2 \mathcal G_\Lambda(u)[\xi] \leq 2\left[\int_{\O_u}|\nabla (\nabla u\cdot \xi)|^2\,dx-\int_{\partial\O_u}|H|(\nabla u\cdot \xi)^2\,d\HH^{d-1}\right]
$$
where $H$ is the mean curvature of $\partial \O_u\cap \text{\rm supp}\,\xi$.\medskip

We proceed by constructing a specific family of vector fields depending on a parameter $\eps \in (0,1)$. Ultimately, the result follows by passing to the limit as $\eps\to 0^+$.
Up to a change of coordinate, it is not restrictive to assume that $x_0$ is of the form $x_0=(x_0',0)$ and that $e_d$ is the interior normal to the free boundary at $x_0$. Now, consider a smooth function $G_\eps \in C^\infty(B_R(x_0))$ such that
$$
\int_{B_R(x_0)\cap \{x_d>0\}}|\nabla G_\eps|^2\,dx = \eps,\qquad \int_{B_R(x_0)\cap \{x_d=0\}}|G_\eps|\,dx' = 1,
$$
where we used the notation $x=(x',x_d)$, with $x' \in \R^{d-1}$.
Now, it is not restrictive to suppose that $B_R(x_0)\cap \partial \O_u$ is the graph of a $C^{2,\alpha}$-function, so that
	$$\O_u\cap B_R(x_0)=\Big\{(x',x_d)\in B_R'(x_0')\times(-R,R)\ :\ x_d>\phi(x')\Big\}.$$
where $\phi:B_R'(x_0')\to(-R,R)$ is a $C^\infty$ smooth function satisfying
$$
\phi(x_0')=|\nabla_{x'}\phi (x_0')|=0\qquad\mbox{with}\qquad D^2_{x'} \phi(x_0') \,\mbox{ diagonal}.
$$
Moreover, by differentiating twice the free boundary condition $u(x',\phi(x'))=0$, for every $x' \in B_R'(x_0')$, and by exploiting the $(-1)$-homogeneity of the second derivatives of $u$, we deduce that in $B_R'(x_0')$
$$
\begin{aligned}
|\nabla_{x'}\phi(x')| &\leq C \sqrt{\Lambda} \norm{D^2_{x'}u}{L^\infty(B_R'(x_0'))}\\
&\leq C \frac{\sqrt{\Lambda}}{R}\norm{D^2_{x'}u}{L^\infty(B_1'(x_0'/R))}\leq C \frac{\sqrt{\Lambda}}{R}\norm{D^2_{x}u}{L^\infty(\partial B_1)}\\
\end{aligned}
$$
with $C>0$ a dimensional constant.
Then, if we consider $g_\eps \in C^\infty_c(B_R(x_0))$ such that  $$
g_\eps(x',x_d) = G_\eps\left(x',x_d-\phi(x')\right),\qquad \mbox{with }(x',x_d)\in B_R(x_0)\cap \O_u,
$$
we infer that
$$
\int_{B_R(x_0)\cap \O_u}|\nabla g_\eps|^2\,dx \leq C \frac{\Lambda}{R^2}\norm{D^2_{x}u}{L^\infty(\partial B_1)}^2\eps,$$ $$\int_{B_R(x_0)\cap \partial \O_u}|g_\eps|\,d\HH^{d-1} \leq C \left(1+\frac{\Lambda}{R^2}\norm{D^2_{x}u}{L^\infty(\partial B_1)}^2\right).
$$
Finally, since $\partial \O_u$ is homogeneous, if we consider
\be\label{e:eta}
\xi(x)=\eps^{d-1}g_\eps(\eps x)\nabla u(x),
\ee
we get
$$
\begin{aligned}
|\partial^2 \mathcal G_\Lambda(u)[\xi]| &\leq 2\Lambda \left|\eps^{2d-2}\int_{B_{\frac{R}{\eps}}\left(x_0/\eps\right)\cap \O_u}|\nabla (g_\eps(\eps x))|^2\,dx-\eps^{2d-2} \int_{B_{\frac{R}{\eps}}\left(x_0/\eps\right)\cap \partial\O_u}|H(x)|g_\eps(\eps x)^2\,d\HH^{d-1}\right|\\
&\leq 2\Lambda\eps^d\left|\int_{B_{R}(x_0)\cap \O_u}|\nabla g_\eps|^2\,dx-\int_{B_{R}(x_0)\cap \partial\O_u}|H(x)|g_\eps^2\,d\HH^{d-1}\right|\\
&\leq 2\Lambda \left(1+\frac{\Lambda}{R^2}\norm{D^2_{x}u}{L^\infty(\partial B_1)}^2\right)\left(1+ \frac{1}{R}\norm{H}{L^\infty(\partial B_1\cap \partial \O_u)}\right)\eps^d \quad \longrightarrow 0^+
\end{aligned}
$$
as $\eps \to 0^+$. Naturally, in the previous inequality we used that the $1$-homogeneity of $u$ implies that $H$ is a $-1$-homogeneous in $\R^d$.\\

Finally, by choosing in \eqref{e:dai} the vector field $\xi \in C^\infty_c(\R^d;\R^d)$ defined as \eqref{e:dai}, we get
$$
\delta^2 \mathcal G_\Lambda(u)[\eta] \geq - \left(\int_{\O_u} \dive\eta\,dx\right)^2 \delta^2 \mathcal{G}_\lambda(u)[\xi] \geq -C \left(\int_{\O_u} \dive\eta\,dx\right)^2 \eps^d
$$
which leads to the stability condition \eqref{e:global-stable-solutions-main}, as $\eps\to 0^+$.
\section{On the Aguilera-Caffarelli-Spruck functional}\label{s:agcafs}
As a consequence of \cref{s:AC} and the analysis in \cite[Section 2]{BMMTV} and \cite[Section 4]{BMMTV}, we finally deduce an estimate on the dimension of the singular part of the boundary of the optimal sets arising in a heat conduction problem studied in \cite{agcasp} by Aguilera, Caffarelli and Spruck.
\subsection{Preliminary results on the optimal shape}\label{s:prel} In
\cite{agcasp} the authors proved the existence of an optimal (open) set $\O\subset D$ that minimizes  \eqref{e:F2.2} and for which the following holds.
\begin{itemize}
\item[$\qquad(\mathcal H1)$] The state function $u_\O, v_\O$ are Lipschitz continuous and non-negative in $D$. 
\item[$\qquad(\mathcal H2)$] There exists $\eps_0>0, r_0>0$ such that
	$$\eps_0|B_r|\leq |B_r(x_0)\cap \Omega| \leq (1-\eps_0)|B_r|,$$
for every ball $B_r(x_0)\subset D$ of radius $r\le r_0$ centered on $\partial\Omega\cap D$. Moreover, for every $x_0\in \partial \Omega\cap D$, $r \in (0,r_0)$, it holds
\be\label{e:ut}
|\{0<u_\O<rt\}\cap B_r(x_0)|\leq Ct|B_r|,
\ee
for every $t>0$; the same estimate \eqref{e:ut} being valid also for the function $v_\O$.
\item[$\qquad(\mathcal H3)$] The functions $u_\O, v_\O$ are non-degenerate, that is, there is a constant $C>0$ for which
	$$\frac{1}{r^{d-1}}\int_{\partial B_r(x)}u_\Omega\,d\HH^{d-1}\ge \frac{1}{C}\,r,\qquad
\frac{1}{r^{d-1}}\int_{\partial B_r(x)}v_\Omega\,d\HH^{d-1}\ge \frac{1}{C}\,r
$$
for every $x\in\overline \Omega$ and every $B_r(x)\subset D$.
\item[$\qquad(\mathcal H4)$] There is a constant $\tilde M>0$ such that, for every $B_{2r}(x_0)\subset D$, we have the bound
	\begin{equation*}
	0\le \bigg|\int_{B_{2r}(x_0)}\nabla u_\Omega\cdot\nabla \varphi\,dx\bigg| + \bigg|\int_{B_{2r}(x_0)}\nabla v_\Omega\cdot\nabla \varphi\,dx\bigg|\le \tilde{M} \|\varphi\|_{L^\infty(B_{2r}(x_0))},
	\end{equation*}
for every $\varphi\in C^{0,1}_c(B_{2r}(x_0))$, where $C^{0,1}_c(B_{2r}(x_0))$ is the space of Lipschitz functions with compact support in $B_{2r}(x_0)$.
	\item[$\qquad(\mathcal H5)$] $\Omega$ is an NTA domain, so the Boundary Harnack Principle holds on $\Omega$.
\end{itemize}

\subsection{Blow-up analysis and dimension reduction}
Using the blow-up analysis from \cite[Section 4]{BMMTV} and by adapting the argument from \cite[Lemma 8.2]{BMMTV} to the stability along volume-preserving variations (see \cref{t:var}), we get the following compactness result.
\begin{lemma}\label{l:blow}
Let $\O$ be an optimal set for \eqref{e:heat-conduction-problem} and let $u:=u_\O$ and $v:=v_\O$ be the state functions from \eqref{e:heat-conduction-u-Omega} and \eqref{e:heat-conduction-v-Omega}. Given $x_0 \in \partial \O\cap D$ and $r>0$, we set, as usual,
$$u_{x_0,r}(x):=\frac{1}{r}u(x_0+r x)\qquad\text{and}\qquad v_{x_0,r}(x):=\frac{1}{r}v(x_0+r x)\,,$$
for a choice of the parameter so that $B_r(x_0)\subset D$.
Then, there exists a sequence $r_k\to 0$ such that
$$\lim_{k\to\infty}u_{x_0,r_k}=u_0\qquad\text{and}\qquad \lim_{k\to\infty}v_{x_0,r_k}=v_0\,$$
locally uniformly in $\R^d$ and strongly in $H^1_{loc}(\R^d)$. Moreover, the following holds:
\begin{enumerate}
  \item[\rm(1)] $u_0$ and $v_0$ are proportional\,\rm;\it
  \item[\rm(2)] there exists $\lambda>0$ such that $\lambda u_0$ is a global stable solution of the one-phase problem along mean-zero variations, in the sense of \cref{def:global-stable-solutions-meanzero}.
\end{enumerate}
\end{lemma}
\begin{proof}
The proof consists of applying the general program from \cite{BMMTV} (see also \cite{mtv.BH, mtv.flat}) by keeping track of the stability condition of \cref{t:var} along volume-preserving variations. Therefore, we simply sketch the key points of this strategy.\\

First, by exploiting assumptions $(\mathcal H1)-(\mathcal H5)$, we get that there exists $r_k\searrow 0^+$ such that $u_{x_0,r_k}$ and $v_{x_0,r_k}$ converge strongly in $H^1_{loc}(\R^d)$, and locally uniformly in $\R^d$, respectively to $u_0$ and $v_0$. Moreover, the sets $\frac1{r_k}(\O-x_0)$ converge to the limit set $\O_0=\{u_0>0\}=\{v_0>0\}$, with respect to the local Hausdorff distance (see \cite[Section 4]{BMMTV} for more details on the blow-up procedure).\\

\noindent \it Step 1. Stability of with respect to blow-up sequences. \rm
First, by the compactness result we deduce that the validity of $(\mathcal H1)-(\mathcal H5)$ implies that the limits $u_0$ and $v_0$ satisfy the assumptions $(\mathcal A1)-(\mathcal A5)$ of \cref{def:global-stable-solutions}. Then, by \cref{t:var}, there exists $\Lambda \geq0 $ such that the functional
$$
\mathcal{F}_\Lambda(\O)=\mathcal{F}(\O) + \Lambda |\O|
$$
satisfies
\begin{equation}\label{e:1}
\delta \mathcal F_\Lambda(\O)[\eta]=0\qquad\text{and}\qquad \delta^2 \mathcal F_\Lambda(\O)\left[\eta - \frac{\int_\O \dive\eta\,dx}{\int_\O \dive\xi\,dx}\xi\right]\ge0\,,
\end{equation}
for every couple $\eta,\xi\in C^\infty_c(\R^d;\R^d)$ of smooth vector fields with disjoint supports and $\int \dive \xi\not=1$ (see \eqref{e:deriv-F} and \cref{t:var} for the definition of $\mathcal F_\Lambda$). Then, by combining the notion of stable solution along volume-preserving variations \eqref{e:1} with the argument in \cite[Lemma 8.2]{BMMTV}, we obtain that the first and the second variations of $\F_\Lambda$ pass to the limit, i.e.
\begin{equation}
\delta \mathcal F_\Lambda(\O_0)[\eta]=0\qquad\text{and}\qquad \delta^2 \mathcal F_\Lambda(\O_0)\left[\eta - \frac{\int_{\O_0} \dive\eta\,dx}{\int_{\O_0} \dive\xi\,dx}\xi\right]\ge0\,.
\end{equation}
Since in this case the domain $\O_0$ is unbounded, we stress that $\delta u_0$ and $\delta v_0$ solve the PDEs \begin{align*}
\begin{aligned}
-\Delta (\delta u_0)=\dive\big((\delta A)\nabla u_0\big)\quad\text{in }\O_0\ ,\qquad \delta u_0\in \dot H^1_0(\O_0)\,;\\
-\Delta (\delta v_0)=\dive\big((\delta A)\nabla v_0\big)\quad\text{in }\O_0\ ,\qquad \delta v_0\in \dot H^1_0(\O_0)\,,
\end{aligned}
\end{align*}
in the sense of \cite[Section 7.1]{BMMTV} (see also \cref{r:deltaA}).\\

\noindent {\it Step 2. $u_0$ and $v_0$ are proportional.}  Since by $(\mathcal{H}5)$ the Boundary Harnack Principle holds on the optimal set $\O$, the ratio of the two state variables is locally H\"{o}lder continuous and positive. Therefore, the blow-up limits previously defined are proportional, and $v_0=Cu_0$ for some $C>0$ (see the ``second blow-up'' in \cite[Section 4.3]{BMMTV}).\\

\noindent \it Step 3. The Lagrange multiplier $\Lambda$ is positive. \rm
We notice that if $\Lambda >0$, then, up to multiplying $u_0$ by a constant, the inequalities \eqref{e:1} correspond precisely to the stability condition \eqref{e:stab.mean} in \cref{def:global-stable-solutions-meanzero}, which concludes the proof of the Lemma.\\

Therefore, we only need to show that $\Lambda\neq0$. Suppose by contradiction that $\Lambda=0$. Then, by exploiting the proportionality of $u_0$ and $v_0$, we have that:
\begin{itemize}
\item[(\rm a)] $u_0\in H^1_\loc(\R^d)$ is a non-negative Lipschitz function satisfying
    $$
    \Delta u_0 =0 \quad \mbox{in }\O_0\,\rm;
    $$
\item[(\rm b)] $u_0$ satisfies the extremality condition
$$
0=\int_{\O_0} \nabla u_0\cdot(\delta A)\nabla u_0  \, dx=
\int_{\O_0} \left(2\nabla u_0 \cdot D\eta(\nabla u_0) - |\nabla u|^2\dive\eta\right)  \, dx\,,
$$
for every smooth vector field with compact support $\eta \in C^\infty_c(\R^d;\R^d)$.
\end{itemize}
Then, by \cite[Proposition 11.4]{V} we get that
$$
|B\setminus \O_0|=0,\quad\mbox{for every ball }B\subset \R^d,
$$
in contradiction with the density estimates $(\mathcal{A}4)$.
\end{proof}	

Although in \cite{agcasp} the authors proved the local $C^\infty$ regularity for the regular part $\text{\rm Reg}(\partial \O_u \cap D)$ of the free boundary, in order to apply the program defined in \cite{BMMTV} it is more convenient to reformulate the notion of ``regular part'' of the free boundary as follows.

\subsection{Proof of \cref{t:heat-conduction-singular}}\label{sub:heat-conduction-main-result} Let now $\Omega$ be the optimal set for \eqref{e:heat-conduction-problem} from \cref{t:ACS}. We notice that the regular and the singular parts of the free boundary $\partial\Omega\cap D$ defined in \cite{agcasp} can be equivalently characterized as follows.
\begin{enumerate}
\item[\rm(1)]The regular part, $\text{\rm Reg}(\partial\Omega)$, is the set of points $x_0\in\partial\Omega\cap D$ at which there exists a blow up limit $u_0,v_0:\R^d\to\R$ of the form
$$u_0(x)=\alpha(x\cdot\nu)_+\quad\text{and}\quad v_0(x)=\beta(x\cdot\nu)_+\,,$$
for some unit vector $\nu\in\R^d$ and some constants $\alpha>0$ and $\beta>0$ such that
$\alpha\beta=\Lambda$\,, where $\Lambda$ is precisely the Lagrange multiplier from \cref{t:var}.\rm
\item[\rm(2)] The remaining part of the free boundary is the singular set:
$$\text{\rm Sing}(\partial\Omega):=(\partial\Omega\cap D)\setminus \text{\rm Reg}(\partial\Omega).$$
\end{enumerate}
Now, proceeding as in \cite[Theorem 8.1]{BMMTV}, we use the epsilon-regularity result \cite{mtv.flat} and the classical argument for estimating the dimension of the singular set via the Federer's reduction principle (see for instance \cite[Section~8]{BMMTV} and \cite[Chapter~10]{V}), we obtain that the dimension of the singular set is preserved under blow-up limits.
We notice that here, as in the penalized case considered in \cite[Appendix A]{BMMTV}, the vectorial nature of the problem \eqref{e:heat-conduction-problem} requires performing a ``double blow-up''.
Indeed, by applying the blow-up analysis consecutively, we first reduce the problem to the scalar case (i.e. the blow-up limits are proportional, \cref{l:blow}).
Then, since the first variation is preserved under blow-up, we obtain that the monotonicity formula of Weiss holds (see \cite[Section 4.4]{BMMTV}) for the first blow-up.
As a consequence, the second blow-up is a $1$-homogeneous function. By the last part of \cref{l:blow}, this function is a global stable critical point of the one-phase problem along mean-zero variations.
At this point, \cref{t:heat-conduction-singular} follows by applying \cref{t:AC}.


\begin{thebibliography}{999}
\bibitem{agalca}
N.~E.~Aguilera, H.~W.~Alt, L.~A.~Caffarelli,
{\it An optimization problem with volume constraint.}
SIAM J. Control Optim. (24) (1986), no. 2, 191–-198.

\bibitem{agcasp} N.~E.~Aguilera, L.~A.~Caffarelli, J.~Spruck,
{\it An optimization problem in heat conduction.}
Ann. Scuola Norm. Sup. Pisa Cl. Sci. (4) {\bf14} (1987), no. 3, 355--387.

\bibitem{altcaf}H.W.~Alt, L.A.~Caffarelli. {\it Existence and regularity for a minimum problem with free boundary.} J. Reine Angew. Math. {\bf325} (1981), 105--144.	

\bibitem{briancon} T.~Brian\c{c}on. {\it Regularity of optimal shapes for the Dirichlet’s energy with volume constraint.} ESAIM,
Control Optim. Calc. Var. {\bf 10} (2004), 99–122.

\bibitem{brianconlamboley} T.~Brian\c{c}on, J.~Lamboley. {\it Regularity of the optimal shape for the first eigenvalue of the laplacian with volume and inclusion constraints.} Ann. Inst. H. Poincar\'e Anal. Non Lin\'eaire {\bf 26} (4) (2009), 1149--1163.



\bibitem{BMV1}G.~Buttazzo, F.P.~Maiale, B.~Velichkov. {\it Shape optimization problems in control form.} Atti Accad. Naz. Lincei Cl. Sci. Fis. Mat. Natur. {\bf32} (3) (2021), 413--435.

\bibitem{BMMTV}G.~Buttazzo, F.P.~Maiale, D.~Mazzoleni, G.~Tortone, B.~Velichkov. {\it Regularity of the optimal sets for a class of integral shape functionals.} (2022) ArXiv preprint.




\bibitem{cjk}L.~A.~Caffarelli, D.~Jerison, C.E.~Kenig. {\it Global energy minimizers for free boundary problems and full regularity in three dimensions.} Contemp. Math. {\bf 350} Amer. Math. Soc., Providence RI (2004), 83--97.




\bibitem{dsj}D.~De~Silva, D.~Jerison. {\it A singular energy minimizing free boundary.} J. Reine Angew. Math. {\bf 635} (2009), 1--21.


\bibitem{js}D.~Jerison, O.~Savin. {\it Some remarks on stability of cones for the one phase free boundary problem.} Geom. Funct. Anal. {\bf25} (2015), 1240--1257.



\bibitem{mtv.flat}F.P.~Maiale, G.~Tortone, B.~Velichkov. {\it Epsilon-regularity for the solutions of a free boundary system.} Rev. Mat. Iberoam. {\bf39} (2023), no. 5, pp. 1947–1972.

\bibitem{mtv.BH}F.P.~Maiale, G.~Tortone, B.~Velichkov. {\it The Boundary Harnack principle on optimal domains.} Ann. Sc. Norm. Super. Pisa Cl. Sci., to appear, DOI: 10.2422/2036-2145.202112\_003.



\bibitem{Maggi} F.~Maggi. {\it Sets of finite perimeter and geometric variational problems: an introduction to Geometric
Measure Theory}. Cambridge University Press 135 (2012)


\bibitem{V}B.~Velichkov. {\it Regularity of the one-phase free boundaries.} volume 28 of \textit{Lecture Notes of Unione Matematica Italiana.} Springer, 2023.

\bibitem{weiss}G.S.~Weiss. {\it Partial regularity for a minimum problem with free boundary.} J. Geom. Anal. {\bf 9} (2) (1999), 317--326.

\end{thebibliography}
\end{document}